\documentclass[preprint,12pt]{elsarticle}
\usepackage{amsthm,amsfonts,amssymb,amscd,amsmath,enumerate,verbatim,calc,graphicx,geometry}
\usepackage[all]{xy}
\newtheorem{theorem}{Theorem}[section]

\newtheorem{proposition}[theorem]{Proposition}
\newtheorem{corollary}[theorem]{Corollary}
\theoremstyle{definition}
\theoremstyle{definitions}
\newtheorem{definition}[theorem]{Definition}

\newtheorem{remark}[theorem]{Remark}
\newtheorem{example}[theorem]{Example}
\theoremstyle{notations}

\theoremstyle{remarks}

\newcommand{\lo}{\longrightarrow}

\newcommand{\wt}{\widetilde}

\newcommand{\al}{\alpha}

\newcommand{\bt}{\beta}

\newcommand{\psp}{\pi_1^{sp}(X,x)}

\journal{ }
\begin{document}
\begin{frontmatter}
\title{Unique Path Lifting from Homotopy Point of View}
\author[affil1]{Mehdi~Tajik}
\ead{mm.tj@stu.um.ac.ir}
\author[affil1]{Behrooz~Mashayekhy\corref{cor1}}
\ead{bmashf@um.ac.ir}
\author[affil2]{Ali~Pakdaman}
\ead{a.pakdaman@gu.ac.ir}
\address[affil1]{Department of Pure Mathematics, Center of Excellence in Analysis on Algebraic Structures, Ferdowsi University of Mashhad,
P.O.Box 1159-91775, Mashhad, Iran.}
\address[affil2]{Department of Mathematics, Faculty of Sciences, Golestan University, P.O.Box 155, Gorgan, Iran.}
\cortext[cor1]{Corresponding author}
\begin{abstract}
The paper is devoted to introduce some notions extending the unique path lifting property from a homotopy viewpoint and to study their roles in the category of fibrations.\ First, we define some homotopical kinds of the unique path lifting property and find all possible relationships between them. Moreover, we supplement the full relationships of these new notions
in the presence of fibrations.\ Second, we deduce some results in the category of fibrations with these notions instead of unique path lifting such as the existence of products and coproducts. Also, we give a brief comparison of these new categories to some categories of the other generalizations of covering maps.\ Finally, we present two subgroups of the fundamental group related to the fibrations with these notions and compare them to the subgroups of the fundamental group related to covering and generalized covering maps.
\end{abstract}

\end{frontmatter}
\section{Introduction}
\subsection{Motivation}
We recall that a map $p:E\lo B$ is called a fibration if it has homotopy lifting property with respect to an arbitrary space.\ A map $p:E\lo B$ is said to have unique path lifting property if given paths $w$ and $w'$ in $E$ such that $p\circ w=p\circ w'$ and $w(0)=w'(0)$, then $w=w'$ (see \cite{SP}).

Fibrations in homotopy theory and fibrations with unique path lifting property, as a generalization of covering spaces, are important.\
In fact, unique path lifting causes a given fibration $p:E\lo B$ has some behaviours similar to covering maps such as injectivity of induced homomorphism $p_*$, uniqueness of lifting of a given map and being homeomorphic of any two fibers \cite{SP}.\ Moreover, unique path lifting has important role in covering theory and some recent generalizations of covering theory in \cite {BR1,BR2,BD2,FZ}. In the absence of unique path lifting, some certain fibrations exist in which some of the above useful properties are available. However, these fibrations lack some of the properties which unique path lifting guaranties them, notably homeomorphicness of fibers.

We would like to generalize unique path lifting in order to preserve some homotopical behaviors of fibrations with unique path lifting.\ In Section 2, we consider unique path lifting problem in the homotopy category of topological spaces by introducing some  various kinds of the unique path lifting property from homotopy point of view. Moreover, we find all possible relationships between them by giving some theorems and examples.
Then in Section 3, we supplement the full relationships between these new notions in the presence of fibrations and also study fibrations with these new unique path lifting properties .

By the \emph{weakly unique path homotopically lifting property (wuphl)} of a map $p:E\lo B$ we mean that if $p\circ w\simeq p\circ w'$\ rel\ $\dot I$, $w(0)=w'(0)$ and $w(1)=w'(1)$, then $w\simeq w'$\ rel\ $\dot I$.\ We will show among other things that a fibration has wuphl if and only if every loop in each of its fibers is nullhomotopic, which is a homotopy analogue of a similar result when we deal with unique path lifting property (see \cite[Theorem 2.2.5]{SP}).\

 In Section 4, we make a new category, Fibwu, in which objects are fibrations with weakly unique path homotopically lifting property and commutative diagrams are morphisms.\ This category has the category of fibrations with unique path lifting property, Fibu, as a subcategory.\ Also, by fixing base space of fibrations, we construct the category Fibwu(B) of fibrations over a space $B$ with weakly unique path homotopically lifting property as objects and commutative triangles as morphisms.\ We show that these new categories have products and coproducts.\ A brief comparison of these new categories to the categories of other generalizations of covering maps is brought at the end of the section.\

 Finally, in the last section, we introduce two subgroups of the fundamental group of a given space $X$, $\pi_1^{fu}(X,x)$ and $\pi_1^{fwu}(X,x)$.\ In fact, these are the intersection of all the image subgroups of fibrations with unique path lifting and fibrations with weakly unique path homotopically lifting over $X$, respectively.\ We find the relationships of these two subgroups with the two famous subgroups of the fundamental group, the Spanier group $\psp$  and the generalized subgroup $\pi_1^{gc}(X,x)$ (see \cite{FRVZ} and \cite{BR1}, resectively).\ As an application, we show that the category Fibu(X) admits a simply connected universal object if and only if $\pi_{1}^{fu}(X,x)=0$.\
\subsection{Preliminaries}
Throughout this paper, all spaces are path connected.\ A \emph{map} $f:X\lo Y$ means a continuous function and $f_*:\pi_1(X,x)\lo\pi_1(Y,y)$ will denote a homomorphism induced by $f$ on fundamental groups when $f(x)=y$.\ Also, by the \emph{image subgroup of $f$} we mean $f_*(\pi_1(X,x))$.\

For given maps $p:E\lo B$ and $f:X\lo B$, the \emph{lifting problem} for $f$ is to determine whether there is a map $f':X\lo E$ such that $f=p\circ f'$.\ A map $p:E\lo B$ is said to have the \emph{homotopy lifting property} with respect to a space $X$ if for given maps $f':X\lo E$ and $F:X\times I\lo B$ with $F\circ J_0=p\circ f'$, where $J_0:X\lo X\times I$ defined by $J_0(x)=(x,0)$, there is a map $F':X\times I\lo E$ such that $F'\circ J_0=f'$ and $p\circ F'=F$.\

If $\al:I\lo X$ is a path from $x_0=\al(0)$ to $x_1=\al(1)$, then $\al^{-1}$ defined by $\al^{-1}(t)=\al(1-t)$ is the inverse path of $\alpha$.\ For $x\in X$, $c_x$ is the constant path at $x$.\ If $\al, \bt:I\lo X$ are two paths with $\al(1)=\bt(0)$, then $\al*\bt$ denotes the usual concatenation of the two paths.\ Also, all homotopies between paths are relative to end points.\

A covering map is a map $p:\widetilde{X}\to X$ such that for every $x\in X$, there exists an open neighborhood $U$ of $x$, such that $p^{-1}(U)$ is a union of disjoint open sets in $\widetilde{X}$, each of which is mapped homeomorphically onto $U$ by $p$.\

%
 The category whose objects are topological spaces and whose morphisms are maps is denoted by Top and, by hTop, we mean the homotopy category of topological spaces.\ By Fib, we mean the category whose objects are fibrations and whose morphisms are commutative diagrams of maps
$$\xymatrix{
E \ar[r]^{h} \ar[d]_{p}
& E'\ar[d]^{p'}  \\
B \ar[r]^{h'} & B'  }$$
where $p:E\lo B$ and $p':E'\lo B'$ are fibrations.\ For a given space $B$, there exists a subcategory of Fib, denoted by Fib(B), whose objects are fibrations with base space $B$ and morphisms are commutative triangles.
$$\xymatrix{
E \ar[r]^{h} \ar[dr]_p & E'\ar[d]^{p'}  \\
 & B  }$$
 If we restrict ourselves to fibrations with unique path lifting, then we get two subcategories of Fib and Fib(B) which we denote them by Fibu and Fibu(B), respectively.
\section{Homotopically Lifting}
Let $p:E\to B$ be a map and $\alpha : I\rightarrow B$ be a path in $B$. A path $\wt{\alpha}:I\to E$ is called a lifting of the path $\alpha$ if $p\circ\wt{\alpha}=\alpha $.
Existence and uniqueness of path liftings are interesting problems in the category of topological spaces, Top.\ We are going to consider the path lifting problems in the homotopy category, hTop.\
\begin{definition}
Let $p:E\to B$ be a map and $\alpha : I\rightarrow B$ be a path in $B$.\ By a homotopically lifting of the path $\alpha$ we mean a path $\wt{\alpha}:I\to E$ with $p\circ\wt{\alpha}\simeq\alpha\ \mathrm{rel}\ \dot{I}$.\
\end{definition}
In the following, we recall the well-known notion {\it unique path lifting} and introduce some various kinds of this notion from homotopy point of view.\
\begin{definition}Let $p:E\to B$ be a map and let $\wt{\alpha}\ and\ \wt{\beta}$ be two arbitrary paths in E.\ Then we say that
\\(i)\ p has unique path lifting property (upl) if
\\$$\wt{\alpha}(0)=\wt{\beta}(0),\ p\circ\wt{\alpha} = p\circ\wt{\beta} \Rightarrow \wt{\alpha} = \wt{\beta}.$$
\\(ii)\ p has homotopically unique path lifting property (hupl) if
\\$$\wt{\alpha}(0)=\wt{\beta}(0),\ p\circ\wt{\alpha} =  p\circ\wt{\beta}\Rightarrow\wt{\alpha} \simeq \wt{\beta}\ \mathrm{rel}\ \dot{I} .$$
\\(iii)\ p has weakly homotopically unique path lifting property (whupl) if
\\$$\wt{\alpha}(0)=\wt{\beta}(0),\ \wt{\alpha}(1)=\wt{\beta}(1),\ p\circ\wt{\alpha} = p\circ\wt{\beta}\Rightarrow\wt{\alpha} \simeq \wt{\beta}\ \mathrm{rel}\ \dot{I}.$$
\\(iv)\ p has unique path homotopically lifting property (uphl) if
\\$$\wt{\alpha}(0)=\wt{\beta}(0),\ p\circ\wt{\alpha} \simeq p\circ\wt{\beta}\ \mathrm{rel}\ \dot{I} \Rightarrow\wt{\alpha} \simeq \wt{\beta}\ \mathrm{rel}\ \dot{I} .$$
\\(v)\ p has weakly unique path homotopically lifting property (wuphl) if
\\$$\wt{\alpha}(0)=\wt{\beta}(0),\ \wt{\alpha}(1)=\wt{\beta}(1),\ p\circ\wt{\alpha} \simeq  p\circ\wt{\beta}\ \mathrm{rel}\ \dot{I}\Rightarrow\wt{\alpha} \simeq \wt{\beta}\ \mathrm{rel}\ \dot{I}. $$
\end{definition}
\begin{example}Every continuous map from a simply connected space to any space has wuphl and whupl.\ Note that every injective map has upl and also, for injective maps, wuphl and uphl are equivalent.\
\end{example}

We recall that for a given pointed space $(X,x)$, $P(X,x)$ is the set of all paths in $X$ starting at $x$.\ Also, we recall that the fundamental groupoid of $X$ is the set of all homotopy classes of paths in $X$ which we denote it by $\Pi X$
$$\Pi X=\{[\al]\ |\ \al:I\lo X\ is\ continuous\}.$$
If $f:X\lo Y$ is a map, then by $Pf:P(X,x)\lo P(Y,y)$ we mean the function given by $Pf(\al)=f\circ \al$ and by $f_*:\Pi X\lo \Pi Y$ we mean the function given by $f_*([\al])=[f\circ\al]$.\
Also, by a slight modification, we define the set of all paths in $X$ starting at $x$
$$\Pi (X,x)=\{[\al]\in \Pi X\ |\ \al(0)=x\}.$$
By a straightforward verification we have the following results.\
\begin{proposition}\label{P2.4}
Let $p:E\lo B$ be a map.\ Then\\
$(i)$ Injectivity of $Pp:P(E,e)\lo P(B,b)$ for any $e\in E$ is equivalent to $p$ having upl.\ \\
$(ii)$ Injectivity of $p_{\ast}: \pi_{1}(E,e)\to \pi_{1}(B,b)$ for any $e\in E$ is equivalent to $p$ having wuphl.\ \\
$(iii)$ Injectivity of $p_{\ast}: \Pi(E,e)\to \Pi(B,b)$ for any $e\in E$ is equivalent to $p$ having uphl.\\
$(iv)$ Injectivity of $p_{\ast}: \pi_{1}(E,e)\to \pi_{1}(B,b)$ for any $e\in E$ implies that $p$ has whupl.\\
$(v)$ Injectivity of $p_{\ast}: \Pi E\to \Pi B$ implies that $p$ has wuphl.\\
$(vi)$ Injectivity of $p_{\ast}: \Pi(E,e)\to \Pi(B,b)$ for any $e\in E$ implies that $p$ has hupl.
\end{proposition}
It is worth to note that the converse implications of (iv) and (vi) do not hold in general (see Example \ref{E2.8}, part (iv)).\ To see that the converse implication of (v) does not hold, consider the first projection  $pr_1:R^2\lo R$ and the two paths $\wt{\al},\wt{\bt}:I\lo R^2$ given by $\wt{\al}(t)=(t,1)$ and $\wt{\bt}(t)=(t,2)$.\ Obviously, $pr_1\circ\wt{\al}=pr_1\circ\wt{\bt}$ while $\wt{\al}$ and $\wt{\bt}$ do not have the same initial and end points.\

In the next proposition, we show that uniqueness and homotopically uniqueness of path lifting are equivalent.\
\begin{proposition} A map $p:E\to B$ has upl if and only if $p$ has hupl.
\end{proposition}
\begin{proof}By definitions, if $p$ has upl, then $p$ has  hupl.\ Let $p$ have hupl and $\wt{\alpha}\ and\ \wt{\beta}$ be two paths in E with $\wt{\alpha}(0)=\wt{\beta}(0)$, $p\circ\wt{\alpha} = p\circ\wt{\beta}$.\ Define, for every $t\in I$, $\wt{\alpha}_{t},\ \wt{\beta}_{t}:I\to E$ by $\wt{\alpha}_{t}(s)=\wt{\alpha}(st)$ and $\wt{\beta}_{t}(s)=\wt{\beta}(st)$.\ Clearly $\wt{\alpha}_{t}(0)=\wt{\beta}_{t}(0)$ and $p\circ\wt{\alpha}_{t} = p\circ\wt{\beta}_{t}$.\ Since $p$ has hupl, we have $\wt{\alpha}_{t} \simeq\ \wt{\beta}_{t}\ \mathrm{rel}\ \dot{I}$ and so $\wt{\alpha}_{t}(1)=\wt{\beta}_{t}(1)$ which implies that $\wt{\alpha}(t)=\wt{\beta}(t)$.\ Hence $\wt{\alpha}=\wt{\beta}$.\
\end{proof}

Using definition and a similar argument as the above, the following results hold.\
\begin{proposition}\label{p2.6} The following implications hold for any map $p:E\to B$.\\
$(i)$ $upl\Rightarrow whupl$.\\\
$(ii)$ $uphl\Rightarrow whupl$.\\\
$(iii)$ $uphl\Rightarrow wuphl$.\\\
$(iv)$ $uphl\Rightarrow upl$.\\\
$(v)$ $wuphl\Rightarrow whupl$.\
\end{proposition}

A map $p:E\lo B$ is said to have the \emph{unique lifting property} with respect to a space $X$ if by given two liftings $f,g:X\lo E$ of the same map (that is $p\circ f=p\circ g$) such that agrees for some point of $X$, we have $f=g$.
Since maps with upl have unique lifting property with respect to path connected spaces \cite [Lemma 2.2.4]{SP}, the following result is a consequence of the implication $uphl\Rightarrow upl$.\
\begin{corollary}If a map has uphl, it has the unique lifting property with respect to path connected spaces.\
\end{corollary}
 Following examples show that the converse of implications in Proposition \ref{p2.6} do not hold.\
\begin{example}\label{E2.8}
\ \\$(i)$ $wuphl\nRightarrow uphl$.\ Let $E=\{{0}\}\times [0,1]\times [0,1]$ and $B=\{{0}\}\times [0,1]\times \{{0}\}$.\ If $p:E\lo B$ is the vertical projection from $E$ onto $B$, then $p$ has wuphl since E is simply connected.\ But $p$ does not have uphl.\ For if $\wt{\alpha},\wt{\beta}:I\lo E$ are two paths with $\wt{\alpha}(t)=(0,0,\frac{t}{2})$ and $\wt{\beta}(t)=(0,0,t)$ and  $\alpha:I\lo B$ is the constant path at $(0,0,0)$, then $\wt{\alpha}(0)=(0,0,0)=\wt{\beta}(0)$ and $p\circ\wt{\alpha}=\alpha=p\circ\wt{\beta}$ while $\wt{\alpha}$ and $\wt{\beta}$ are not path homotopic.\\\
$(ii)$  $whupl\nRightarrow uphl$.\ Using the same example as (i).\\\
$(iii)$ $whupl\nRightarrow upl$.\ Using the same example as (i).\\\
$(iv)$ $upl\nRightarrow uphl$.\ Let $E=\{{(x,y,2)\in R^{3}}\} -\{{(0,0,2)}\}$, $B=\{{(x,y,0)}\in R^{3}\}$ and $p:E\lo B$ be the vertical projection.\
\\ $(v)$ $whupl\nRightarrow wuphl$.\ Using the same example as (iv).\
\end{example}

Note that among the results of this section, there are no relationship between upl and wuphl.\ In the following example, we show that neither of the two properties implies the other.\

\begin{example}The map example introduced in Example \ref{E2.8}, (iv), has upl but does not have wuphl.\ Conversely, let $p:{\{0}\}\times [0,1]\to {\{0}\}$ be the constant map, then p has wuphl but it does not have upl.\
\end{example}

We can summarize the results of this section in the following diagram.\\\\\\\\

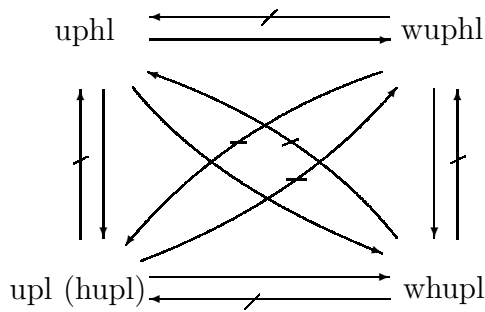
\unitlength 1mm 
\linethickness{0.4pt}
\ifx\plotpoint\undefined\newsavebox{\plotpoint}\fi 
\begin{picture}(93.192,62.102)(0,0)
\put(38.573,58.081){uphl}
\put(84.655,58.23){wuphl}
\put(84.952,23.743){whupl}
\put(32.629,23.588){upl (hupl)}
\put(45.103,51.49){\vector(0,-1){19.814}}
\put(89.092,51.542){\vector(0,-1){19.814}}
\put(42.055,31.571){\vector(0,1){19.919}}
\put(92.193,31.624){\vector(0,1){19.919}}
\multiput(43.053,42.555)(-.0670345,-.0326207){29}{\line(-1,0){.0670345}}
\multiput(93.192,42.608)(-.067069,-.0326207){29}{\line(-1,0){.067069}}
\put(51.147,58.144){\vector(1,0){32.112}}
\put(83.096,61.136){\vector(-1,0){31.886}}
\multiput(68.157,62.102)(-.03345,-.03343333){60}{\line(-1,0){.03345}}
\multiput(65.897,24.252)(-.03345,-.03345){60}{\line(-1,0){.03345}}
\put(83.259,23.582){\vector(-1,0){32.112}}
\put(51.2,26.578){\vector(1,0){31.954}}
\put(84.063,51.765){\vector(3,4){.07}}\qbezier(50.126,28.703)(71.657,36.921)(84.063,51.765)
\put(48.063,30.765){\vector(-3,-4){.07}}\qbezier(82.063,53.828)(61.438,46.984)(48.063,30.765)
\put(51.001,53.828){\vector(-3,1){.07}}\qbezier(84.126,31.703)(72.813,46.14)(51.001,53.828)
\put(82.126,29.703){\vector(3,-1){.07}}\qbezier(49.001,51.765)(59.313,38.796)(82.126,29.703)
\multiput(71.063,44.89)(-.0841154,-.0336538){26}{\line(-1,0){.0841154}}
\put(72.063,39.515){\line(-1,0){2.563}}
\put(64.126,44.453){\line(-1,0){2.063}}
\end{picture}

\section{Fibrations and Homotopically Liftings}
 In the classic book of Spanier \cite [Chapter 2]{SP} one can find a considerable studies on fibrations with unique path lifting property.\ In this section, we intend to study and compare fibrations with the various kinds of homotopically unique path lifting properties introduced in Section 2.\

 Examples 2.8\ (iv)\ and 2.9\ show that the two implications $upl\ (hupl)\ \Rightarrow uphl$ and
$upl\ (hupl)\ \Rightarrow wuphl$ do not hold in general.\ In the following proposition we show that these two implications hold with the presence of fibrations.\
\begin{proposition}\label{P3.1} For fibrations the following implications hold.\
\\(i)\ upl (hupl) $ \Rightarrow uphl$
\\(ii)\ upl (hupl) $ \Rightarrow wuphl$
\begin{proof}For (i) see \cite [Lemma 2.3.3]{SP}.\ Part (ii) comes from the definitions and part (i).\
\end{proof}
\end{proposition}

The following corollary is a consequence of the above result and Proposition 2.6 (iv).\
\begin{corollary}\label{C3.2} For fibrations, upl (hupl) and uphl are equivalent.\
\end{corollary}

In the following example, we show that the converse of Proposition 3.1 (ii) does not hold. Note that fibrations with unique path lifting which are generalizations of covering maps, has no nonconstant path in their fibers. In fact, for fibrations, this is equivalent to the unique path lifting (see \cite [Theorem 2.2.5]{SP}).
\begin{example}\label{E3.3}  Let $p:X\times Y\lo X$ be the projection which is a fibration, where $Y$ is a non-singleton simply connected space.\ If $x\in X$, then the fiber over $x$,  $p^{-1}(x)=\{{x}\}\times Y$ is homeomorphic to Y and so every fiber has a nonconstant path which implies that $p$ does not have upl.\ To show that $p$ has wuphl, let $\wt{\alpha}, \wt{\beta}:I\lo X\times Y$ be two homotopically lifting of a path $\alpha:I\to X$ with $\wt{\alpha}(0)=\wt{\beta}(0)=(x_{0},y_{0})\in p^{-1}(x_{0})$, where $x_{0}=\alpha(0)$ and $\wt{\alpha}(1)=\wt{\beta}(1)$.\ Then $\wt{\alpha}=(\wt{\alpha}_{1},\wt{\alpha}_{2})$, $\wt{\beta}=(\wt{\beta}_{1},\wt{\beta}_{2})$, where $\wt{\alpha}_{1},\wt{\beta}_{1}$ are paths in $X$ with $\wt{\alpha}_{1}(0)=\wt{\beta}_{1}(0)=x_{0}$ and $\wt{\alpha}_{2}$, $\wt{\beta}_{2}$ are paths in $Y$ with $\wt{\alpha}_{2}(0)=\wt{\beta}_{2}(0)=y_{0}$.\ If $p\circ\wt{\alpha}\simeq p\circ\wt{\beta}\simeq \alpha$ rel\ $\dot{I}$, then $\wt{\alpha}_{1}\simeq\wt{\beta}_{1}$ rel\ $\dot{I}$.\ Since $\wt{\alpha}(1)=\wt{\beta}(1)$, we have $\wt{\alpha}_{2}(1)=\wt{\beta}_{2}(1)$ and hence $\wt{\alpha}_{2}\ast\wt{\beta}_{2}^{-1}$ is a loop in $Y$ at $y_{0}$.\ Simply connectedness of $Y$ implies that $\wt{\alpha}_{2}\simeq\wt{\beta}_{2}$ rel\ $\dot{I}$.\ Thus $\wt{\alpha}\simeq\wt{\beta}$\ rel\ $\dot{I}$.\
\end{example}
In the following theorem, we show that considering unique path lifting problem in the homotopy category makes all paths in fibers homotopically constant, i.e.\ nullhomotopic.
\begin{theorem}\label{T3.4}
If $p:E\lo B$ is a fibration, then $p$ has wuphl if and only if every loop in each fiber is nullhomotopic.\
\end{theorem}
\begin{proof}
First, assume that $p$ has wuphl and $\alpha:I\lo p^{-1}(b_{0})$ is a loop in the fiber over $b_{0}$ in $E$ which implies that $p\circ\alpha=c_{b_{0}}$, where $c_{b_{0}}$ is the constant path at $b_{0}$.\ Also, we have $p\circ c_{\alpha(0)}=c_{b_{0}}$, $\alpha(0)=c_{\alpha(0)}(0)$ and $\alpha(1)=c_{\alpha(0)}(1)$.\ Then $\alpha\simeq c_{\alpha(0)}$\ rel\ $\dot{I}$,\ since $p$ has wuphl which implies that $\alpha$ is nullhomotopic.

Conversely, let $\wt{\alpha},\wt{\beta}:I\lo E$ be two paths with $\wt{\alpha}(0)=\wt{\beta}(0)$, $\wt{\alpha}(1)=\wt{\beta}(1)$ and $p\circ\wt{\al}\simeq p\circ\wt{\bt}$,\ rel\ $\dot{I}$.\ We show that $\wt{\alpha}\simeq\wt{\beta}$\ rel\ $\dot{I}$.\ Let $\gamma:=\wt{\alpha}^{-1}*\wt{\beta}$ which is a loop at $\wt{\alpha}(1)=\wt{\beta}(1)$.\ Put $\tilde{x}_0=\wt{\al}(1)$ and $x_0=p(\tilde{x}_0)$, then we have $$p\circ\gamma=p\circ(\wt{\alpha}^{-1}*\wt{\beta})=(p\circ\wt{\alpha}^{-1})*(p\circ\wt{\beta})=(p\circ\wt{\alpha})^{-1}*(p\circ\wt{\beta})\simeq c_{x_0}\ \mathrm{rel}\ \dot{I}.$$
Let $F:p\circ\gamma\simeq c_{x_0}$\ rel\ $\dot{I}$.\ Since $p$ is a fibration, there exists a homotopy $\wt{F}:I\times I\lo E$ such that $p\circ\wt{F}=F$ and $\wt{F}\circ J_0=\gamma$.\ Put $f:=\wt{F}(0,-), g:=\wt{F}(- ,1)$ and $h:=\wt{F}(1,- )$.\ Then $f,g$ and $h$ are paths in $E$ with $f(1)=g(0)$ and $g(1)=h(1)=h^{-1}(0)$.\ Define $\eta:=f*g*h^{-1}$ which is a loop at $\wt{\alpha}(1)$ since  $$\eta(0)=f(0)=\wt{F}(0,0)=\wt{F}\circ J_0(0)=\gamma(0)=\wt{\alpha}^{-1}(0)=\wt{\alpha}(1),$$
$$\eta(1)=h^{-1}(1)=h(0)=\wt{F}(1,0)=\wt{F}\circ J_0(1)=\gamma(1)=\wt{\beta}(1)=\wt{\alpha}(1).$$
Also, we have $$p\circ\eta=(p\circ f)*(p\circ g)*(p\circ h)^{-1}=$$
$$F(0,-)*F(-,1)*(F(1,-))^{-1} =c_{x_0}*c_{x_0}*(c_{x_0})^{-1}$$
 which implies that $\eta$ is a loop in a fiber.\ By assumption, $\eta$ is nullhomotopic, namely, $\eta\simeq c_{\tilde{x}_{0}}$\ rel\ $\dot{I}$.\ On the other hand, if we show that $\eta*\gamma\simeq c_{\tilde{x}_{0}}$ rel\ $\dot{I}$, then $\gamma\simeq c_{\tilde{x}_{0}}$ rel\ $\dot{I}$ and so $\wt{\alpha}\simeq\wt{\beta}$,\ rel\ $\dot{I}$.\ Since $\wt{F}$ is a homotopy such that $\wt{F}|_{I\times {0}}=\gamma$ and $\wt{F}|_{(\{{0}\}\times I)\cup (I\times \{{1}\})\cup(\{{1}\}\times I)}=\eta$, it is obvious that $\eta\ast \gamma \simeq c_{\tilde{x}_{0}}$\ rel\ $\dot{I}.$
\end{proof}
By a similar proof to the above, we can replace wuphl with whupl.\
\begin{theorem}\label{T3.5}
A fibration $p:E\lo B$ has whupl if only if every loop in each fiber is nullhomotopic.\
\end{theorem}
\begin{corollary}\label{C3.6}
If $p:E\to B$ is a fibration, then whupl and wuphl are equivalent.\
\end{corollary}
\begin{remark}
Note that the converse of Corollary \ref{C3.6} does not necessarily hold.\ As an example, if $p:\{{*}\} \lo I$ is the constant map $*\mapsto 0$, then $p$ has wuphl and whupl but $p$ is not a fibration.\ To see this, let $\wt{f}:X\lo {\{\ast}\}$ defined by $x\mapsto *$ and $F:X\times I\to I$ defined by $ F(x,t)=t$, then $p\circ\wt{f}=F\circ J_0$.\ But there is no map $\wt{F}:X\times I \lo \{{*}\}$ such that $p\circ\wt{F}=F$ because $p\circ\wt{F}(x,0.5)=p(*)=0$ but $F(x,0.5)=0.5$.\
\end{remark}
It is known that if $p:E\to B$ is a fibration with upl, then the induced homomorphism by $p$, $p_{*}:\pi_{1}(E,e_{0})\to \pi_{1}(B,b_{0})$ is a monomorphism \cite [Theorem 2.3.4]{SP}.\ By Proposition \ref{P2.4} and Corollary \ref{C3.6}, we have a similar result for fibrations with whupl.\
\begin{corollary} If $p: (E,e_{0})\lo (B,b_{0})$ is a fibration, then whupl is equivalent to injectivity of\ $p_{\ast}: \pi_{1}(E,e_{0})\lo \pi_{1}(B,b_{0})$.\
\end{corollary}
\begin{corollary}\label{T3.9} For a fibration $p:E\lo B$ with wuphl and path connected fibers, the induced homomorphism $p_*:\pi_1(E,e_0)\lo \pi_1(B,b_0)$ is an isomorphism.\
\end{corollary}
\begin{proof}
Let $[\al]\in\pi_1(B,b_0)$ and $\wt{\al}$ be a lifting of $\al$ starting at $e_0$, then $\wt{\al}(1)\in p^{-1}(b_0)$.\ Assume that $\lambda$ is a path in $p^{-1}(b_0)$ from $\wt{\al}(1)$ to $e_0$, then $[\wt{\al}*\lambda]\in\pi_1(E,e_0)$ and $p_*([\wt{\al}*\lambda])=[\al*c_{b_0}]=[\al]$.\ Hence $p_*$ is onto.\ Injectivity of $p_*$ comes from Proposition \ref{P2.4} (ii).\
\end{proof}

Note that path connectedness of fibers is essential in the previous theorem.\ For example, let $p$ be the exponential map $R \rightarrow S^{1}$ which is a covering map.\ Clearly, fibers are discrete and we know that $p$ is a fibration with upl and so by Proposition \ref{P3.1} (ii), $p$ has wuphl, but $p_{*}$ is not an isomorphism.\\

 The results of this section can now be summarized in the following diagram.\\\\\\
\unitlength 1mm 
\linethickness{0.4pt}
\ifx\plotpoint\undefined\newsavebox{\plotpoint}\fi 
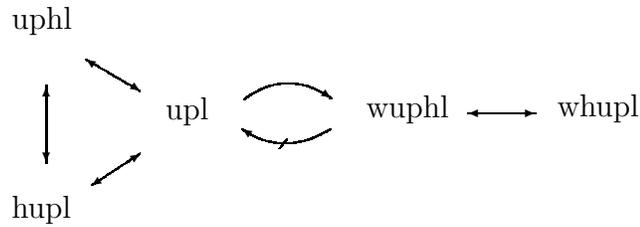
\begin{picture}(87.83,50.254)(0,0)
\put(15.305,50.254){uphl}
\put(15.305,25.004){hupl}
\put(35.778,38.179){upl}
\put(32.413,41.855){\vector(3,-2){.07}}\put(25.163,46.105){\vector(-3,2){.07}}\multiput(25.163,46.105)(.057539683,-.033730159){126}{\line(1,0){.057539683}}
\put(25.913,29.355){\vector(-3,-2){.07}}\put(32.413,33.605){\vector(3,2){.07}}\multiput(32.413,33.605)(-.051587302,-.033730159){126}{\line(-1,0){.051587302}}
\put(19.913,32.355){\vector(0,-1){.07}}\put(19.913,42.605){\vector(0,1){.07}}\put(19.913,42.605){\line(0,-1){10.25}}
\put(62.462,38.443){wuphl}
\put(87.83,38.591){whupl}
\put(57.806,40.924){\vector(3,-2){.07}}\qbezier(46.227,40.835)(51.928,45.034)(57.806,40.924)
\put(84.88,38.947){\vector(1,0){.07}}\put(75.961,38.947){\vector(-1,0){.07}}\put(75.961,38.947){\line(1,0){8.9191}}
\put(45.933,36.865){\vector(-3,2){.07}}\qbezier(57.677,36.865)(51.359,33)(45.933,36.865)
\multiput(51.879,35.528)(-.0335663,-.0431567){31}{\line(0,-1){.0431567}}
\end{picture}

The following two diagrams give a comparison of relationship between the five kinds of the unique paths liftings.\\\\\

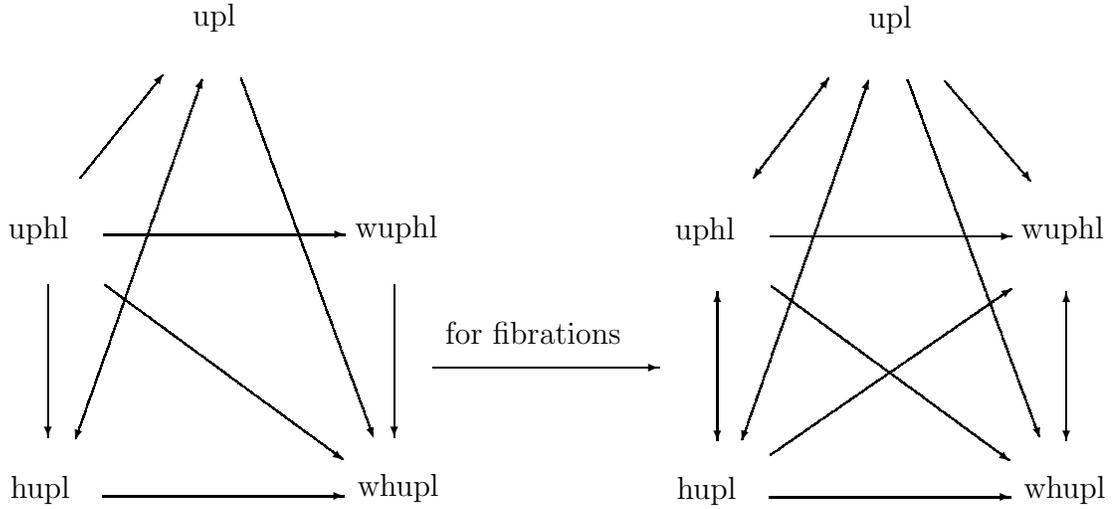
\unitlength 1mm 
\linethickness{0.4pt}
\ifx\plotpoint\undefined\newsavebox{\plotpoint}\fi 
\begin{picture}(145.596,71.487)(0,0)
\put(29.457,71.487){upl}
\put(119.314,71.338){upl}
\put(4.944,43.243){uphl}
\put(93.54,43.095){uphl}
\put(5.241,8.756){hupl}
\put(93.838,8.607){hupl}
\put(51.026,43.392){wuphl}
\put(139.622,43.243){wuphl}
\put(51.323,8.905){whupl}
\put(139.92,8.756){whupl}
\put(13.89,16.635){\vector(-1,-3){.07}}\put(30.688,64.352){\vector(1,3){.07}}\multiput(30.688,64.352)(-.0337309237,-.0958172691){498}{\line(0,-1){.0958172691}}
\put(102.487,16.486){\vector(-1,-3){.07}}\put(119.284,64.203){\vector(1,3){.07}}\multiput(119.284,64.203)(-.0337289157,-.0958172691){498}{\line(0,-1){.0958172691}}
\put(99.216,16.337){\vector(0,-1){.07}}\put(99.216,36.108){\vector(0,1){.07}}\put(99.216,36.108){\line(0,-1){19.771}}
\put(145.596,16.337){\vector(0,-1){.07}}\put(145.596,36.108){\vector(0,1){.07}}\put(145.596,36.108){\line(0,-1){19.771}}
\put(140.839,50.824){\vector(3,-4){.07}}\multiput(129.393,64.203)(.0336647059,-.03935){340}{\line(0,-1){.03935}}
\put(53.432,16.932){\vector(1,-3){.07}}\multiput(35.891,64.5)(.0337326923,-.0914769231){520}{\line(0,-1){.0914769231}}
\put(142.028,16.783){\vector(1,-3){.07}}\multiput(124.487,64.352)(.0337326923,-.0914788462){520}{\line(0,-1){.0914788462}}
\put(17.607,43.689){\vector(1,0){32.109}}
\put(106.203,43.541){\vector(1,0){32.109}}
\put(17.458,8.905){\vector(1,0){32.109}}
\put(106.054,8.756){\vector(1,0){32.109}}
\put(49.418,13.81){\vector(4,-3){.07}}\multiput(17.755,37)(.04602180233,-.03370639535){688}{\line(1,0){.04602180233}}
\put(138.015,13.662){\vector(4,-3){.07}}\multiput(106.352,36.851)(.04602180233,-.03370494186){688}{\line(1,0){.04602180233}}
\put(138.312,36.554){\vector(3,2){.07}}\multiput(106.203,14.405)(.0488721461,.0337123288){657}{\line(1,0){.0488721461}}
\put(61.386,26.068){\vector(1,0){30.062}}
\put(63.068,29.221){for fibrations}
\put(10.25,37){\vector(0,-1){20.25}}
\put(56.25,37){\vector(0,-1){20.25}}
\put(114,64.5){\vector(3,4){.07}}\put(104,51.25){\vector(-3,-4){.07}}\multiput(104,51.25)(.0336700337,.0446127946){297}{\line(0,1){.0446127946}}
\put(25.5,65){\vector(3,4){.07}}\multiput(14.5,51.25)(.0336391437,.0420489297){327}{\line(0,1){.0420489297}}
\end{picture}
It is well known that in fibrations, fibers have the same homotopy type and in fibrations with upl and path connected base space, every two fibers are homeomorphic (see \cite[Lemma 2.3.8]{SP}).\ In the following example, we show this fact fails if we replace upl with wuphl (whupl).\
\begin{example}
Let $E=\{(x,y)\in R^2|x\geq 0,y\geq 0, y\leq 1-x\}$, $B=[0,1]$ and $p:E\lo B$ be the projection on the first component which is clearly a map.\ For given maps $F:X\times I\lo B$ and $f:X\lo E$ with $F\circ J_0=p\circ f$, define $\wt{F}:X\times I\lo E$ by $\wt{F}(x,t)=(F(x,t),pr_2\circ f)$ which is continuous and it shows that $p$ is a fibration.\ But $p^{-1}(0)={\{0}\}\times I$ and $p^{-1}(1)=\{(1,0)\}$ which imply that $p$ does not have upl, but $p$ has wuphl and its fibers are not necessarily homeomorphic.\
 \end{example}
Another different influence of upl and wuphl on the fibrations is uniqueness of the lifted homotopy as follows.\
\begin{proposition}Let $p:E\to B$ be a fibration.\ Then $p$ has upl if and only if it has unique homotopy lifting property, namely, every homotopy in $B$ can be lifted uniquely to $E$.\
\end{proposition}
\begin{proof} Let $p$ be a fibration with unique homotopy lifting property, $\widetilde{f}:{\{*}\}\to E$ be defined by $\widetilde{f}(*)=e_0$, $\alpha$ be a path in $B$ starting at $b_{0}:=p(e_{0})$ and $F:{\{*}\}\times I\to B$ be defined by $F(*,t)=\alpha(t)$.\ Then $p\circ \widetilde{f}(*)=b_{0}=\alpha(0)=F(*,0)=F\circ J_0(*)$.\ Since $p$ is a fibration, there is
$\widetilde{F}:{\{*}\}\times I\to E$ with $p\circ\widetilde{F}=F, \widetilde{F}\circ J_0=\widetilde{f}$.\ Define $\widetilde{\alpha}(t)=\widetilde{F}(*,t)$, then $p\circ\widetilde{\alpha}=p\circ\widetilde{F}=F=\alpha$,\  $\widetilde{\alpha}(0)=\widetilde{F}(*,0)=\widetilde{f}(*)=e_{0}$, and so $\widetilde{\alpha}$ is a lifting of $\alpha$ beginning at $e_{0}$.\ Let $\widetilde{\beta}$ be another lifting of $\alpha$ beginning at $e_{0}$, then by defining $\widetilde{G}:{\{\ast}\}\times I\to E$ by $\widetilde{G}(*,t)=\widetilde{\beta}(t)$, we have $p\circ\widetilde{G}(*,t)=p\circ\widetilde{\beta}(t)=\alpha(t)=F(*,t)$ and $\widetilde{G}\circ J_0(*)=\widetilde{G}(*,0)=\widetilde{\beta}(0)=e_{0}=\widetilde{f}(*)$.\ Uniqueness of homotopy lifting implies that $\widetilde{F}=\widetilde{G}$ and hence $\widetilde{F}(*,t)=\widetilde{G}(*,t)$ which implies that $\widetilde{\alpha}(t)=\widetilde{\beta}(t)$.\

 Conversely, let $p$ be a fibration with upl and $\widetilde{f}:Y\to E$, $F:Y\times I\to B$ be two maps with $p\circ\widetilde{f}=F\circ J_0$.\ Also, let $\widetilde{F}, \widetilde{G}:Y\times I\to E$ be two maps with $p\circ\widetilde{F}=p\circ\widetilde{G}=F$, and $\widetilde{F}\circ J_0=\widetilde{G}\circ J_0=\widetilde{f}$.\ For an arbitrary fixed $y\in Y$, let $\alpha(t)=\widetilde{F}(y,t)$ and $\beta(t)=\widetilde{G}(y,t)$, then $p\circ \alpha(t)=p\circ \widetilde{F}(y,t)=F(y,t)$ and $p\circ \beta(t)=p\circ \widetilde{G}(y,t)=F(y,t)$.\ Also, $$\alpha(0)=\widetilde{F}(y,0)=\widetilde{F}\circ J_0(y)=\widetilde{G}\circ J_0(y)=\widetilde{G}(y,0)=\beta(0).$$ Since $p$ has upl, we have $\alpha(t)=\beta(t)$ and hence $\widetilde{F}(y,t)=\widetilde{G}(y,t)$ which implies that $\widetilde{F}=\widetilde{G}$.\
\end{proof}
\begin{proposition} A fibration $p:E\to B$ has wuphl if it has homotopically unique homotopy lifting property, namely, for every topological space $Y$, any homotopy $F:Y\times I\to B$ and  every map $\widetilde{f}:Y\to E$ with $p\circ \widetilde{f}=F\circ J_0$, if there exist homotopies $\widetilde{F},\widetilde{G}:Y\times I\to E$ such that $p\circ \widetilde{F}=F$, $\widetilde{F}\circ J_0=\widetilde{f}$, $p\circ \widetilde{G}=F$ and $\widetilde{G}\circ J_0=\widetilde{f}$, then $\widetilde{F}\simeq \widetilde{G}$, rel ${\{y_0}\}\times \dot{I}$, for a fixed $y_0\in Y$.\
\end{proposition}
\begin{proof}By Corollary 3.6, it is enough to prove that $p$ has whupl.\ Let $\alpha$ be a path in $B$ from $b_0$ to $b_1$ and $\widetilde{\alpha}, \widetilde{\beta}:I\to E$ be two liftings of $\alpha$ from $e_0$ to $e_1$.\ Also, assume that $F:{\{*}\}\times I\to B$ is defined by $F(*,t)=\alpha(t)$ and $\widetilde{f}:{\{*}\}\to E$ is defined by $\widetilde{f}(*)=e_0$.\ Then $p\circ \widetilde{f}(*)=e_{0}=\alpha(0)=F(*,0)=F\circ J_0(*)$.\ Let  $\widetilde{F}, \widetilde{G}:{\{*}\}\times I\to E$ be two map such that $\widetilde{F}(*,t)=\widetilde{\alpha}(t)$ and $\widetilde{G}(*,t)=\widetilde{\beta}(t)$.\ Then $p\circ \widetilde{F}(*,t)=p\circ \widetilde{\alpha}(t)=\alpha(t)=F(*,t)$ and $\widetilde{F}\circ J_0(*)=\widetilde{F}(*,0)=\widetilde{\alpha}(0)=e_0=\widetilde{f}(*)$ and also, $p\circ \widetilde{G}(*,t)=p\circ\widetilde{\beta}(t)=\alpha(t)=F(*,t)$ and $\widetilde{G}\circ J_0(*)=\widetilde{G}(*,0)=\widetilde{\beta}(0)=e_0=\widetilde{f}(*)$.\ By assumption, there exists $H_1:{\{*}\}\times I\times I\to E$ such that $H_1:\widetilde{F}\simeq \widetilde{G}$ rel ${\{\ast}\}\times \dot{I}$.\ Define $H:I\times I\to E$ by $H(t,s)=H_1(\ast,t,s)$.\ It is easy to see that $H:\widetilde{\alpha}\simeq \widetilde{\beta}$ rel\ $\dot{I}$.\
\end{proof}

Note that the converse of the above proposition does not hold, in general.\ Let $p:{\{0}\}\times I\to {\{0}\}$ be the projection, $F:I\times I\to {\{0}\}$ be the constant homotopy $F(t,s)=0$ and $\widetilde{f}:I\to {\{0}\}\times I$ be defined by $\widetilde{f}(t)=(0,\frac{1}{2})$.\ Since the only fiber of $p$ is simply connected, $p$ is a fibration with wuphl.\ Now, let $\widetilde{F}$ and $\widetilde{G}:I\times I\to {\{0}\}\times I$ be two homotopies defined by $\widetilde{F}(t,s)=(0,\frac{1-s}{2})$ and $\widetilde{G}(t,s)=(0,\frac{1+s}{2})$, respectively.\ Then $p\circ\widetilde{F}=F$, $\widetilde{F}\circ J_0=\widetilde{f}$, $p\circ\widetilde{G}=F$ and $\widetilde{G}\circ J_0=\widetilde{f}$.\ Note that $\widetilde{F}$ is not homotopic to $\widetilde{G}$ relative to ${\{0}\}\times \dot{I}$.\

\section{Categorical Viewpoints}
Topological spaces as objects and fibrations with upl as morphisms make a category. Also, fibrations with upl and commutative diagram between them and fibrations with upl over a based space $B$ and commutative triangles between them are two categories which have product and coproduct (see \cite[Section 2.2]{SP}).\
In this section, we state some categorical properties of fibrations with wuphl.\
\begin{proposition}\label{P4.1} \ \\
i) Composition of two maps with wuphl is a map with wuphl.\\\
ii) Composition of two fibrations with wuphl is a fibration with wuphl.\
\end{proposition}
\begin{proof}Part (i) comes from the definition and part (ii) is a consequence of Theorem \ref{T3.4}.\
\end{proof}
By the above proposition, there is a category whose objects are fibrations with wuphl and whose morphisms are commutative diagrams of maps
$$\xymatrix{
E \ar[r]^{h} \ar[d]_{p}
& E'\ar[d]^{p'}  \\
B \ar[r]^{h'} & B'  }$$
where $p:E\lo B$ and $p':E'\lo B'$ are fibrations with wuphl.\ We denote this category by Fibwu which has Fibu as a subcategory.\ Also, for a given space $B$, there exists another subcategory of Fibwu, denoted by Fibwu(B), whose objects are fibrations with wuphl which have $B$ as the base space and whose morphisms are commutative triangles
$$\xymatrix{
E \ar[r]^{h} \ar[dr]_p & E'\ar[d]^{p'}  \\
 & B}$$
 Obviously, Fibu(B) is a subcategory of Fibwu(B).\
Note that in the above diagram although $p,p'$ are fibrations, $h$ is not necessarily a fibration.\ By the following proposition and example, we show that upl property of $p,p'$ is sufficient for $h$ being a fibration with upl, while wuphl property is not.\
\begin{proposition}\label{P4.2}Every morphism in the category $\mathrm{Fibu(B)}$ is a fibration with upl.
\end{proposition}
\begin{proof}Consider a morphism in Fibu(B) as follows:
$$\xymatrix{
E \ar[r]^{h} \ar[dr]_p & E'\ar[d]^{p'}  \\
 & B}$$
Let $Z$ be a space, $\widetilde{f}:Z\to E$ be a map and $F:Z\times I\to E'$ be a homotopy such that $h\circ \widetilde{f}=F\circ J_0$.\ Then $p'\circ h\circ \widetilde{f}=p'\circ F\circ J_0$ and so $p\circ \widetilde{f}=(p'\circ F)\circ J_0$.\ Since $p$ is a fibration, there is a homotopy $\widetilde{G}:Z\times I\to E$ such that $p\circ \widetilde{G}=p'\circ F$ and $\widetilde{G}\circ J_0=\widetilde{f}$.\ Hence $p'\circ h\circ \widetilde{G}=p'\circ F$ and $h\circ \widetilde{G}\circ J_0=h\circ \widetilde{f}=F\circ J_0$.\ For an arbitrary fix $z\in Z$, we have $p'\circ h\circ \widetilde{G}(z,-)=p'\circ F(z,-)$ and $h\circ \widetilde{G}(z,0)=F(z,0)$.\ Since $p'$ has upl, we have $h\circ\widetilde{G}(z,-)=F(z,-)$ and since $z$ is arbitrary, $h\circ\widetilde{G}=F$.\ Therefore $h$ is a fibration.\ Moreover, $h$ has upl.\ To show this, let $\widetilde{\alpha}$ and $\widetilde{\beta}$ be two paths in $E$ beginning from the same point and $h\circ\widetilde{\alpha}=h\circ\wt{\beta}$.\ Then $p'\circ h\circ \widetilde{\alpha}=p'\circ h\circ \wt{\beta}$ and so $p\circ\widetilde{\alpha}=p\circ\wt{\beta}$.\ Since $p$ has upl, we have $\widetilde{\alpha}=\wt{\beta}$.\
\end{proof}
\begin{example}\label{Ex4.3}
Let $CS^1$ be the cone over $S^1$, $S^1\times I/(z,1)\sim (z',1)$.\ Then $p:CS^1\lo \{[(z,1)]\}$ and $p':I\lo \{[(z,1)]\}$ are fibrations with wuphl.\ Define $h:CS^1\lo I$ by $h([(z,t)])=t$ for every $z\in S^1$ and any $t\in I$.\ Obviously $p'\circ h=p$ but $h$ is not a fibration since its fibers do not have the same homotopy type, for example $h^{-1}(1)={\{[(1,1)]}\}$ while $h^{-1}(0.5)={\{[(z,0.5)]|z\in S^1}\}$ which is homeomorphic to $S^1$ .\
\end{example}
It is known that any family of objects in the categories Fibu and Fibu(B) has a product and coproduct (see \cite [pp. 69-70]{SP}). Now, we are going to show that this fact holds in the category Fibwu and Fibwu(B).\
\begin{proposition}\label{P4.4}The product of fibrations with wuphl is a fibration with wuphl.
\end{proposition}
\begin{proof}Since the product of fibrations is a fibration, it is sufficient to show that every loop in each fiber of product of fibrations is nullhomotopic.\ But this is because of that a loop in a fiber of a product of fibrations is a product of loops each of which is in a fiber of a fibration with wuphl.\
\end{proof}

 To show that Fibwu(B) has the products, let us to recall the Whitney sum of fibrations.\ If ${\{p_{j}:E_{j}\to B|j\in J}\}$ is an indexed collection of fibrations with wuphl over the space $B$, define $$\oplus_{B,J}E_j={\{(e_{j})_{j}\in \sqcap_j{E_j}|e_{j}\in E_{j},\ \mathrm{and}\ p_{j}(e_{j})=p_{i}(e_{i}),\ \mathrm{for}\ i,j\in J}\}$$
and also define $$\oplus_{B,J}p_{j}:\oplus_{B,J}E_j\to B$$ $$(e_{j})_{j}\rightarrowtail p_{j}(e_{j}).$$
Since $(\oplus_{B,J}p_{j})^{-1}(b)={\{(e_j)_j\in \sqcap_j{E_j}|p_j(e_j)=b,\ \mathrm{for}\ j\in J}\}$, the fibers of $\oplus_{B,J}p_{j}$ are the product of the fibers of $p_j$ and so we can deduce that $\oplus_{B,J}p_{j}$ is a fibration with wuphl.\
\begin{proposition}\label{P4.5} Let ${\{p_{j}:E_{j}\to B|j\in J}\}$ be an indexed collection of fibrations with wuphl on the space $B$.\ Then $\oplus_{B,J}p_{j}$ is a fibration with wuphl.\
\end{proposition}
The following result is a consequence of Propositions \ref{P4.4}, \ref{P4.5}.
\begin{theorem}\label{T4.6}
The categories $\mathrm{Fibwu}$ and $\mathrm{Fibwu(B)}$ have products.\
\end{theorem}
Suppose ${\{p_{j}:E_{j}\to B_j|j\in J}\}$ is an indexed collection of objects in Fibwu and $\sqcup_j E_j$ is the disjoint union of $E_j$'s. Then $q:\sqcup_j E_j\lo \sqcup_j B_j$ given by $q|_{E_j}=p_j$ is a fibration and since a fiber of $q$ is a fiber of one of ${p_j}^{,}s$, every loop in the fibers of $q$ is nullhomotopic and hence $q$ has wuphl.\ Also, if ${\{p_{j}:E_{j}\to B|j\in J}\}$ is an indexed collection of objects in Fibwu(B), then $q':\sqcup_{j} E_j\lo B$ given by $q'|_{E_j}=p_j$ is also a fibration.\ Note that fibers of $q'$ are the disjoint union of fibers of ${p_j}^{,}s$ and so every loop in fibers of $q'$ is nullhomotopic.\ Hence $q'$ has wuphl.\ Therefore, we have the following result.\
\begin{theorem}\label{T4.7}
The categories $\mathrm{Fibwu}$ and $\mathrm{Fibwu(B)}$ have coproducts.\
\end{theorem}
If $f:X\to B$ is a map, we define a functor from Fibwu(B) to Fibwu(X) and we show that this functor preserves the universal objects. Recall that if $p:E\to B$ is a fibration, then the projection $f^{\ast}p:X\times_{B}E\to X$ is a fibration which is called the fibration induced from $p$ by $f$ (see \cite[page 98]{SP}). Now, we have the following result.
\begin{proposition}\label{P4.8}
If $p:E\lo B$ is a fibration with wuphl and $f:X\lo B$ is a map, then $f^{\ast}p$ is a fibration with wuphl.\
\end{proposition}
\begin{proof}
let $\alpha,\beta$\ be paths in $X\times_{B}E$ with the same initial point and the same end point. Then $\alpha=(\alpha_{1},\alpha_{2})$ and $\beta=(\beta_{1},\beta_{2})$, where $\alpha_{1},\beta_{1}$ and $\alpha_{2},\beta_{2}$\ are paths in $X$ and $E$, respectively.\ Also, since $\alpha(0)=\beta(0)$, $\alpha(1)=\beta(1)$, we have $\alpha_{1}(0)=\beta_{1}(0)$, $\alpha_{2}(1)=\beta_{2}(1)$.\ Assume $(f^{\ast}p)\circ\alpha\simeq (f^{\ast}p)\circ\beta\ \mathrm{rel}\ \dot{I}$.\ By definition $\alpha_{1}\simeq \beta_{1}\ \mathrm{rel}\ \dot{I}$.\ Hence $f\circ\alpha_{1}\simeq f\circ\beta_{1}\ \mathrm{rel}\ \dot{I}$ and since $(\alpha_{1}(t),\alpha_{2}(t))\in X\times_{B}E$ for all $t\in I$, we have $p\circ\alpha_{2}\simeq p\circ\beta_{2}\ \mathrm{rel}\ \dot{I}$.\ But $p$ has wuphl and therefore $\alpha_{2}\simeq \beta_{2}\ \mathrm{rel}\ \dot{I}$.\ Hence $\alpha\simeq \beta\ \mathrm{rel}\ \dot{I}$ which implies that $f^{\ast}p$ has wuphl.
\end{proof}
We know that $f^{\ast}: \mathrm{Fib(B)}\to \mathrm{Fib(X)}$\ is a functor. Thus, by the above proposition, we have the following result.\
\begin{theorem}\label{T4.9}
For any map $f:X\lo B$, $f^{\ast}: \mathrm{Fibwu(B)}\to \mathrm{Fibwu(X)}$ is a functor.
\end{theorem}
 \begin{proposition}\label{P4.100}
    If $f:X\to B$ and $p:E\to B$ are two objects in $\mathrm{Fibwu(B)}$, then the projection $q_2:X\times _BE\to E$ is an object in $\mathrm{Fibwu(E)}$.
\end{proposition}
\begin{proof}
Consider two maps $\widetilde{f}:Z\to X\times _BE$ and $F:Z\times I\to E$ with $q_2\circ\widetilde{f}=F\circ J_0$. Then $\widetilde{f}(z)=(pr_{1}\circ\widetilde{f}(z),F(z,0))$ and $f\circ pr_{1}\circ\widetilde{f}(z)=p\circ F(z,0)$. Let $G:=p\circ F$. Then $f\circ pr_{1}\circ\widetilde{f}=G\circ J_0$ and since $f$ is a fibration, there exists a map $\widetilde{G}:Z\times I\to X$ such that $f\circ \widetilde{G}=G$ and $\widetilde{G}\circ J_0=pr_{1}\circ\widetilde{f}$. Hence $f\circ\widetilde{G}=p\circ F$ and so we can define a map $\widetilde{F}:Z\times I\to X\times _BE$ by $\widetilde{F}(z,t)=(\widetilde{G}(z,t),F(z,t))$. Therefore $q_2\circ\widetilde{F}=F$ and $\widetilde{F}\circ J_0=\widetilde{f}$. A similar proof to Proposition \ref{P4.8} shows that $q_2$ has wuphl.
\end{proof}
  \begin{proposition}\label{P4.10}
   Let\ $f:X\to B$ and $p:E\to B$ be two objects in $\mathrm{Fibu(B)}$ (or $\mathrm{Fibwu(B)}$) such that $p$ is a universal object.\ Then $f^{\ast}p:X\times _BE\to X$ is a universal object in $\mathrm{Fibu(X)}$ (or $\mathrm{Fibwu(X)}$).
  \begin{proof}
    Let $g:E'\to X$ be an object in $\mathrm{Fibu(X)}$. Then $p':=f\circ g:E'\to B$ is an object in $\mathrm{Fibu(B)}$ and so universality of $p$ follows that there exists a unique morphism $h:E\to E'$ such that $p'\circ h=p$.\ Since $p$ and $f\circ (f^{\ast}p)$ are fibrations with upl, using Proposition \ref{P4.2},\ the projection $q_2$ is a fibration with upl and so $h\circ q_2$ is a fibration with upl.\ Note that $p'\circ h\circ q_2=p\circ q_2=f\circ (f^{\ast}p)$ and $p'\circ q'_2=f\circ (f^{\ast}p')$ where $q'_2:X\times _BE'\to E'$ is the projection.\ Therefore, the universality of the pullback $X\times _BE'$ follows that there exists a morphism $k:X\times _BE\to X\times _BE'$ such that $f^{\ast}p'\circ  k=f^{\ast}p$.\ Define $t=q'_2\circ k$, then $t$ is a fibration with upl and $g\circ t=g\circ q_2'\circ k=f^{\ast}p'\circ k=f^{\ast}p$. By a similar argument to the above and using Proposition \ref{P4.100}, we have the same result for Fibwu(X).
  \end{proof}
\end{proposition}
\begin{remark}\label{R4.11}
  Recently, Fischer and Zastrow \cite {FZ} and Brazas \cite{BR1,BR2} have introduced two new categories, the category of generalized coverings and the category of semicoverings  over a given space $X$, denoted by\ $\mathrm{GCov(X)}$ and $\mathrm{SCov(X)}$, respectively.\ A generalized covering map is a surjection map $p:\widetilde{X}\to X$ with a path connected and locally path connected total space such that for every  path connected and locally path connected space $Y$, any $\tilde{x}\in \widetilde{X}$, and any map $f:(Y,y)\to (X,p(\tilde{x}))$ with $f_{\ast}\pi_{1}(Y,y)\subseteq p_{\ast}\pi_{1}(\widetilde{X},\tilde{x})$, there exists a unique map $\widetilde{f}:(Y,y)\to (\widetilde{X},\tilde{x})$ such that $p\circ \widetilde{f}=f$, (see \cite{BR1,FZ}).\ Also, a semicovering map is a local homeomorphism which has upl and path lifting property (see \cite[Corollary 2.1]{K}).\ The category of covering spaces of $X$, $\mathrm{Cov(X)}$ is a subcategory of\ $\mathrm{GCov(X)}$ and $\mathrm{SCov(X)}$.\ Note that these categories are not equivalent to $\mathrm{Fibu(X)}$ and $\mathrm{Fibwu(X)}$.\ For comparing these categories, the following diagram summarizes a number of implications of relations between classical coverings and their generalizations.\ According to the enumeration of the implications in the following diagram, for each arrow a reference or a proof is given.\ The label $(1, \Rightarrow)$ means, that an argument is to be given, why this implication is true, while
$(1, \nLeftarrow)$ means, that an argument is to be given, why the converse of this implication does not hold in general.
\end{remark}
\ \\\\
\unitlength 1mm 
\linethickness{0.4pt}
\ifx\plotpoint\undefined\newsavebox{\plotpoint}\fi 
\begin{picture}(106.806,97.05)(0,0)
\put(58.377,89.04){connected}
\put(15.508,57.723){connected}
\put(57.978,60.094){Covering}
\put(81.377,79.675){(1)}
\put(44.816,79.987){(2)}
\put(45.345,39.985){(3)}
\put(82.191,40.29){(4)}
\put(106.806,59.859){(5)}
\put(63.877,97.05){(6)}
\put(20.065,62.372){(7)}
\put(62.73,22.289){(8)}
\put(96.219,96.022){Fibration}
\put(96.813,90.542){with upl}
\put(16.746,96.299){Generalized}
\put(19.855,89.799){covering}
\put(96.248,31.045){Fibration }
\put(100.042,26.413){with}
\put(99.15,22.4){wuphl}
\put(92.015,93.947){\vector(-1,0){51.731}}
\put(105.096,82.947){\vector(0,-1){42.663}}
\put(95.137,36.271){\vector(1,-1){.07}}\multiput(73.731,54.852)(.0388493648,-.033722323){551}{\line(1,0){.0388493648}}
\put(36.122,35.825){\vector(-4,-3){.07}}\multiput(58.122,54.704)(-.0392857143,-.0337125){560}{\line(-1,0){.0392857143}}
\put(36.122,86.218){\vector(-4,3){.07}}\multiput(58.122,67.636)(-.0399274047,.0337241379){551}{\line(-1,0){.0399274047}}
\put(95.285,85.92){\vector(1,1){.07}}\multiput(73.731,67.19)(.0387661871,.0336870504){556}{\line(1,0){.0387661871}}
\put(90.082,32.703){\vector(4,-1){.07}}\put(34.041,82.501){\vector(-1,4){.07}}\qbezier(34.041,82.501)(44.372,43.48)(90.082,32.703)
\put(90.231,60.019){(9)}
\put(45.304,60.031){(10)}
\multiput(66.021,28.814)(-.03471429,-.03314286){56}{\line(-1,0){.03471429}}
\put(21.875,30.25){Semi}
\put(19.25,22.625){covering}
\put(27,40.125){\vector(0,1){42.375}}
\put(48.014,57.156){\line(0,-1){2.973}}
\put(49.501,55.67){\line(-1,0){2.973}}
\put(41.325,32.852){\vector(-4,-1){.07}}\qbezier(97.961,83.393)(87.63,42.366)(41.325,32.852)
\put(92.015,27.946){\vector(-1,0){51.879}}
\multiput(86.515,57.082)(-.1321342,-.0330336){27}{\line(-1,0){.1321342}}
\end{picture}\\
$(1, \Rightarrow)$:\ Follows from Theorems 2.2.2 and 3.2.2 of \cite{SP}.\
\\ $(1, \nLeftarrow)$:\ Let $p:S^1\times N\to S^1$ be defined by $p(z,n)=z^n$.\ Then the restriction of $p$ to the $n$-th componenet, namely, $p_n:S^1\times {\{n}\}\to S^1$ with $p_{n}(z,n)=p(z,n)$ is a covering map and so is a fibration with upl.\
Therefore, by Theorem 2.3.2 of \cite{SP}, $p$ is a fibration.\ Moreover, it is easy to see that $p$ has upl, but $p$ is not a covering map (see \cite [Example 3.8]{BR2}).\
\\ $(2, \Rightarrow)$:\ Refer to \cite{SP}.\
\\ $(2, \nLeftarrow)$:\ Because every generalized universal covering is a generalized covering and using Example 4.15 of \cite{FZ}, a generalized universal covering is not necessarily a covering map.\
\\ $(3, \Rightarrow)$:\ Since a covering map is a local homeomorphism which has path lifting and upl.\
\\ $(3, \nLeftarrow)$:\ The same counterexample as for (1).\
\\ $(4, \Rightarrow)$:\ Follows from (1) and Proposition \ref{P3.1} (ii).\
\\ $(4, \nLeftarrow)$:\ The same counterexample as for (1).\
\\ $(5, \Rightarrow)$:\ It is Proposition \ref{P3.1} (ii).\
\\ $(5, \nLeftarrow)$:\ It is Example \ref{E3.3}.\
\\ $(6, \Rightarrow)$:\ Follows from Theorem 2.4.5 of \cite{SP}.\
\\ $(6, \nLeftarrow)$:\ Similar to $(2, \nLeftarrow)$, Example 4.15 of \cite{FZ} is a generalized universal covering which is not a fibration (with upl).\
\\ $(7, \Rightarrow)$:\ See page 9 of \cite{FZ1}.
\\$(8)$: If ``fibration with wuphl'' $\Rightarrow$ ``semicovering'', then by Proposition \ref{P3.1} (ii)``fibration with upl" $\Rightarrow$ ``semicovering", which contradicts to (9).\
\\ $(9)$:\ Let $p:E\times ({\{0}\}\cup {\{\frac{1}{n}|n\in N}\})\to E$ be the trivial bundle, then $p$ is a fibration with upl.\ But since $p$ is not a local homeomorphism, $p$ is not a semicovering.
\\ $(10)$:\ By $(6, \nLeftarrow)$ we have ``generalized covering" $\nRightarrow$ ``fibration (with wuphl)". Also, ``fibration with wuphl" $\nRightarrow$ ``generalized covering" because in the otherwise since a generalized covering map has upl, we have ``fibration with wuphl" $\Rightarrow$ ``fibration with upl", which is a contradiction (see Example \ref{E3.3}).
%

\section{Some Fibration Subgroups}
In this section, we introduce some normal subgroups of the fundamental group of a given space $X$ related to its fibrations. Then we compare them with the other well known subgroups of the fundamental group of $X$.\
\begin{definition}\label{DE5.1}
Let $X$ be a space and $x_0\in X$.\
\\ (i)\ By the fu-subgroup of $\pi(X,x_0)$ we mean the intersection of all the image subgroups of fibrations over $X$ with upl.\ We denote it by $\pi_{1}^{fu}(X,x_{0})$.\
\\ (ii)\ By the fwu-subgroup of $\pi_{1}(X,x_0)$ we mean the intersection of all the image subgroups of fibrations over $X$ with wuphl.\ We denote it by $\pi_{1}^{fwu}(X,x_{0})$.\
\end{definition}
\begin{proposition}For a given space $X$ and $x_0\in X$, we have $$\label{P5.2}
\pi_{1}^{fwu}(X,x_{0})\unlhd \pi_{1}^{fu}(X,x_{0})\unlhd \pi_{1}(X,x_0).$$
\end{proposition}
\begin{proof}
 Obviously, $\pi_{1}^{fu}(X,x_{0})$ and $\pi_{1}^{fwu}(X,x_{0})$ are subgroups of $\pi_{1}(X,x_0)$ and by Proposition \ref{P3.1},  $\pi_{1}^{fwu}(X,x_{0})\subseteq \pi_{1}^{fu}(X,x_{0})$.\ We show that they are normal subgroups of $\pi_{1}(X,x_0)$.\ Let $[\alpha]\in \pi_{1}(X,x_0)$, $[\beta]\in \pi_{1}^{fu}(X,x_{0})$ (or $\pi_{1}^{fwu}(X,x_{0})$) and $H$ be an arbitrary image subgroup of a fibration with upl (wuphl) $p$ over $X$, namely, $H=p_{*}\pi_{1}(\widetilde{X},\tilde{x})$,\ where $\tilde{x}\in p^{-1}(x_0)$.\ Let $\widetilde{\alpha}$ be a lifting of $\alpha$ at $\tilde{x}$.\ Since $[\beta] \in \pi_{1}^{fu}(X,x_{0})\ (or\ \pi_{1}^{fwu}(X,x_{0}))$ and $\widetilde{\alpha}(1)\in p^{-1}(x_0)$,\ $[\beta]\in p_{*}\pi_{1}(\widetilde{X},\widetilde{\alpha}(1))$ and so there is a loop $\widetilde{\beta}$ at $\widetilde{\alpha}(1)$ such that $\widetilde{\beta}$ is a homotopically lifting of $\beta$.\ Thus $\widetilde{\alpha}\ast \widetilde{\beta}\ast \widetilde{\alpha}^{-1}$ is a loop and a homotopically lifting of $\alpha\ast \beta\ast \alpha^{-1}$ at $\tilde x$ which implies that $[\alpha\ast\beta\ast \alpha^{-1}]\in H$.\
\end{proof}

Let $S={\{p_{j}:E_{j}\to B|j\in J}\}$ be an indexed collection of fibrations with upl over $B$. We know that their product, $\sqcap_{j\in J}p_j:\sqcap_jE_{j}\to \sqcap_jB$ is a fibration with upl, \cite[Theorem 2.2.7]{SP}. Also $q:=\bigtriangleup^{*}(\sqcap_{j}p_{j}):B\times_{\sqcap_{j}B}(\sqcap_{j}E_{j})\to B$ is a fibration with upl, where $\bigtriangleup:B\to \sqcap_{j}B$ is the diagonal map $\bigtriangleup(b)=(b)_{j}$ (see \cite[page 98]{SP}).
\begin{theorem}\label{T5.3}
For a given space $B$ and every $b\in B$, $\pi_{1}^{fu}(B,b)$ is the image subgroup of a fibration with upl over $B$.
\end{theorem}
\begin{proof}
Let {\{$H_{j}|j\in J$}\} be the family of image subgroups of fibrations with upl over $B$.\ For every $j\in J$, there is a fibration with upl $p_{j}:{E}_{j}\to B$ such that $p_{j*}\pi_{1}({E}_{j},{e}_{j})=H_{j}$, for an ${e}_{j}\in p_j^{-1}(b)$.\
 Fix ${e}_{j}$ as the base point of ${E}_j$ and let $p:=\sqcap_{j\in J}p_{j}:\sqcap_{j}{E}_{j}\to \sqcap_{j}B$.\ Then $\bigtriangleup^{*}p:\bigtriangleup^*(\sqcap_{j}{E}_{j})\to B$ is a fibration with upl,\ where $\bigtriangleup^*(\sqcap_{j}{E}_{j}):=B\times_{\sqcap_{j}B}(\sqcap_{j}E_{j})$.\ We show that the image of $(\bigtriangleup^{*}p)_{*}$ is $\cap_{j\in J}H_{j}$.\
Let $[\beta]\in \pi_{1}(\bigtriangleup^*(\sqcap_{j}E_{j}),(b,\sqcap_{j}{e_{j}}))$.\ Then $\beta=(\alpha, \sqcap_{j}\gamma_{j})$, where $\alpha$ is a loop in $B$ at $b$ and for every $j\in J$, $\gamma_{j}$ is a loop in $E_{j}$ at ${e_{j}}$.\ By definition of pullback, $$\bigtriangleup \circ\alpha=(\sqcap_{j}p_{j})\circ(\sqcap_{j}\gamma_{j})=\sqcap_{j}(p_{j}\circ\gamma_{j})$$ which implies that $p_{j}\circ\gamma_{j}=\alpha,$ for any $j\in J$.\ Hence we have $$p_{j*}[\gamma_{j}]=[p_{j}\circ\gamma_{j}]=[\alpha]\Rightarrow[\alpha]\in p_{j*}\pi_{1}({E}_{j},e_j)=H_{j}\Rightarrow[\alpha]\in \cap_{j}H_{j}.$$
Therefore  $(\bigtriangleup^{*}p)_{*}([\beta])=[(\bigtriangleup^{*}p)\circ\beta]=
[(\bigtriangleup^{*}(\sqcap_{j}p_{j}))\circ(\alpha,\sqcap_{j}\gamma_{j})]=[\alpha]$ and hence
 $(\bigtriangleup^{*}p)_{*}\pi_{1}(\bigtriangleup^*(\sqcap_{j}{E}_{j}))\subseteq \cap_{j}H_{j}$.\

  Conversely, let $[\alpha]\in \cap_{j}H_{j}$.\ Then for every $j\in J$, $[\alpha]\in H_{j}=p_{j*}\pi_{1}({E}_{j},{e}_{j})$, for an ${e}_j\in p_{j}^{-1}(b)$.\ For every $j\in J$, there is a loop $\delta_j$ in ${E}_j$ at ${e}_j$ such that $p_j\circ\delta_j\simeq \alpha$ rel\ $\dot{I}$.\ Let $\gamma_j$ be the lifting of $\alpha$ at ${e}_j$ by $p_j$.\ Since $p_j$ is a fibration with upl property, by Proposition \ref{P3.1} (i), $p_{j}$ has uphl and hence $\gamma_j\simeq \delta_j$ rel\ $\dot{I}$ and so $\gamma_j$ is a loop at ${e}_j$.\ Therefore, if we
   put $\beta:=(\alpha,\sqcap_{j}\gamma_{j})$, then $\beta$ is a loop in $\bigtriangleup^*(\sqcap_{j}{E}_{j})$
   since
    $$(\sqcap_{j}p_{j})\circ(\sqcap_{j}\gamma_{j})=\sqcap_{j}(p_{j}\circ\gamma_j)
   =\sqcap_{j}\alpha=\bigtriangleup \circ\alpha$$
    and
     $$(\bigtriangleup^{*}p)_{\ast}([\beta])=[(\bigtriangleup^{*}p)\circ\beta]=
[(\bigtriangleup^{*}p)\circ(\alpha,\sqcap_{j}\gamma_{j})]=[\alpha].$$
 This implies that $\cap_{j}H_{j}\subseteq (\bigtriangleup^{*}p)_{*}\pi_{1}(\bigtriangleup^*(\sqcap_{j}{E}_{j}),(b,\sqcap_{j}{e_{j}})).$
\end{proof}

For an open covering $\mathcal {U}$ of a given space $X$ and $x_{0}\in X$, $\pi(\mathcal {U},x_0)$, the Spanier subgroup with respect to $\mathcal {U}$, is the subgroup of $\pi_1(X,x_0)$ consisting of all homotopy classes of loops that can be represented by a product of the following type $$\prod\limits_{j=1}^{n}\alpha_j\ast\beta_j\ast\alpha_{j}^{-1},$$
 where the $\alpha_j$'s are arbitrary paths starting at the base point $x_0$ and each $\beta_j$ is a loop inside one of the neighborhoods $U_i\in \mathcal {U}$.\ Spanier \cite{SP} used this subgroup for classification of covering spaces of a given space.\ In fact, for every open cover $\mathcal {U}$ of $X$, there exists a covering map $p:\widetilde{X}_{\mathcal {U}}\to X$ such that $p_{\ast}\pi_{1}(\widetilde{X}_{\mathcal {U}},\tilde{x}_0)=\pi(\mathcal {U},x_0)$ and conversely, for every covering map $p:\widetilde{X}\to X$, there exists an open cover $\mathcal {U}$ of $X$ such that  $p_{\ast}\pi_{1}(\widetilde{X},\tilde{x}_0)=\pi(\mathcal {U},x_0)$ (see \cite[Theorems 2.5.12-13]{SP}).\ The Spanier group of a given space $X$, $\pi_{1}^{sp}(X,x_{0})$, which is introduced in \cite{FRVZ} is the intersection of all $\pi(\mathcal {U},x_0)$, for every open cover $\mathcal {U}$ of $X$.\ Mashayekhy et.al \cite{MPT1} used the Spanier group for the existence of some universal coverings of spaces with bad local behaviour. They showed in \cite{MPT1} that if $p:\widetilde{X}\to X$ is a categorical universal covering of $X$, then $p_{\ast}\pi_{1}(\widetilde{X},\tilde{x}_0)=\pi_{1}^{sp}(X,x_{0})$.\ But the existence of such categorical universal covering is not possible in general and we need  $X$ has some local properties which are introduced in \cite{MPT2}.\ Note that these local conditions are not necessary when we work with fibrations with upl.\ In the following theorem, we will compare these subgroups.\
\begin{proposition}\label{P5.4}
If $X$ is a connected and locally path connected space, then $$\pi_{1}^{fwu}(X,x_{0})\subseteq \pi_{1}^{fu}(X,x_{0})\subseteq \pi_{1}^{sp}(X,x_{0}).$$
\end{proposition}
\begin{proof}
  The left inclusion holds by Proposition \ref{P3.1} (ii).\ For the right inclusion, let $\mathcal {U}$ be an open cover of $X$.\ Using \cite[Theorems 2.5.13]{SP} there exsits a covering map $p:\widetilde{X}_{\mathcal {U}}\to X$ with $p_{\ast}\pi_{1}(\widetilde{X}_{\mathcal {U}},\tilde{x}_0)=\pi(\mathcal {U},x_0)$.\ Since every covering map is a fibration with upl, we have $\pi_{1}^{fu}(X,x_{0})\subseteq \pi(\mathcal {U},x_0)$.\ Since $\mathcal {U}$ is arbitrary we can conclude that $ \pi_{1}^{fu}(X,x_{0})\subseteq \pi_{1}^{sp}(X,x_{0})$.\
\end{proof}

 Brazas \cite{BR1} has introduced the subgroup $\pi_{1}^{gc}(X,x)$ which is the intersection of all the image subgroups of generalized covering maps of $X$.\ It is shown that $\pi_{1}^{gc}(X,x)$ is a generalized covering subgroup of $\pi_1(X,x)$ (see \cite [Theorem 2.3.15]{BR1}). Since every fibration with upl is a generalized covering map \cite{SP}, we have the following result.\
\begin{proposition}\label{P5.5}For a given space $X$ and $x_0\in X$, we have
$$\pi_{1}^{gc}(X,x_0)\subseteq \pi_{1}^{fu}(X,x_0)\subseteq \pi_{1}^{sp}(X,x_0).$$
\end{proposition}
\begin{remark}
  Note that by Remark \ref{R4.11}, there is no relationship between generalized coverings and fibrations with wuphl in general which implies that there is no inclusion relation  between $\pi_{1}^{gc}(X,x)$ and $\pi_{1}^{fwu}(X,x)$.
\end{remark}

Let $p:\widetilde{X}\to X$ be a simply connected universal object in the category Fibu(X), i.e, $\pi_{1}(\widetilde{X},\tilde{x})=0$.\ Then since $\pi_{1}^{fu}(X,x)\subseteq  p_{\ast}\pi_{1}(\widetilde{X},\tilde{x})$, we have $\pi_{1}^{fu}(X,x)=0$.\
Conversely, let $\pi_{1}^{fu}(X,x)=0$.\ We know that the category Fibu(X) has a universal object $p:(\widetilde{X},\tilde{x})\to (X,x)$ \cite [page 84]{SP}.\ If $p':(\widetilde{Y},\tilde{y})\to (X,x)$ is an arbitrary object in Fibu(X), then there is an object $q:(\widetilde{X},\tilde{x})\to (\widetilde{Y},\tilde{y})$ such that $p'\circ q=p$.\ Therefore
 $$p_{\ast}\pi_{1}(\widetilde{X},\tilde{x})=(p'\circ q)_{\ast}\pi_{1}(\widetilde{X},\tilde{x})=p'_{\ast}\circ q_{\ast}(\pi_{1}(\widetilde{X},\tilde{x}))\subseteq p'_{\ast}\pi_1(\widetilde{Y},\tilde{y})$$
 which implies that
 $$ p_{\ast}\pi_{1}(\widetilde{X},\tilde{x})\subseteq \bigcap\ {\{p'_{\ast}\pi_1(\widetilde{Y},\tilde{y})|\ p':(\widetilde{Y},\tilde{y})\to (X,x)\ is\ an\ object\ of\ \mathrm{Fibu(X)}}\}$$\
Hence
 $ p_{\ast}\pi_{1}(\widetilde{X},\tilde{x})\subseteq\pi_{1}^{fu}(X,x)=0$ and so $p_{\ast}\pi_{1}(\widetilde{X},\tilde{x})=0.$\\
Since $p_\ast$ is a monomorphism, $\pi_{1}(\widetilde{X},\tilde{x})=0$ and hence $\widetilde {X}$ is simply connected.\ Thus we have the following result.
\begin{theorem}\label{T5.7}
For a given space $X$ and $x\in X$, the category $\mathrm{Fibu(X)}$ admits a simply connected universal object if and only if $\pi_{1}^{fu}(X,x)=0$.\
\end{theorem}

\end{document}